\documentclass[letterpaper, 10pt]{amsart}

\usepackage{mymacros}

\begin{document}

\title{On Poincar\'e Duality for Orbifolds}
\author{Dmytro Yeroshkin}
\address{Syracuse University, Department of Mathematics, 215 Carnegie Building, Syracuse, NY 13244, USA}
\email{dyeroshk@syr.edu}

\begin{abstract}
In this paper we address the relation between the orbifold fundamental group and the topology of the underlying space. In particular, under the assumption that $\pi_1^{orb}(\calO)=\pi_1(|\calO|)$, we prove Poincar\'e Duality for orbifolds of dimension 4 and 5.
\end{abstract}

\maketitle
\section{Introduction}

When studying new manifolds, one of the first properties one considers is the cohomology ring. In particular, one uses the established tools that restrict the behavior of the cohomology ring to simplify this task. One of the best known tools for this is Poincar\'e Duality.

First introduced by Satake as V-manifolds (\cite{SVM}), and later studied in depth by Thurston (\cite{TGT3M}), orbifolds have gained prominence in recent years. See for example \cite{GWZCoho}, \cite{GVZPosC} and \cite{DP2} where orbifolds are used in the construction of the newest example of a manifold with positive sectional curvature. However, many of the tools invented for manifolds do not carry over fully. In particular, while Satake (\cite{SVM}) showed that orbifolds satisfy $\bbR$-valued Poincar\'e Duality, it is easy to come up with examples where $\bbZ$-valued Poincar\'e Duality fails. The focus of this paper is to study the properties of $\bbZ$-valued cohomology for orbifolds with particular focus on dimension 4 and 5.

The choice of dimension is dictated by the fact that in dimension 3 and below all orientable orbifolds are topological manifolds, and starting in dimension 6 the number of obstructions to Poincar\'e Duality increases (see Remark \ref{Rmk:HighDim}). The theorem below provides the foundation for using $\pi_1^{orb}$ to bound the failure of Poincar\'e Duality.

\begin{mthm}\label{Mthm:Pi1Orb}
Let $\calO^n$ be a compact orientable orbifold without boundary, then, the following diagram commutes:

\begin{center}
\begin{tikzpicture}
\node (OP1) at (0,2) {$\pi_1^{orb}(\calO)$};
\node (Coh) at (3,2) {$H^{n-1}(|\calO|;\bbZ)$};
\node (Pi1) at (0,0) {$\pi_1(|\calO|)$};
\node (Hom) at (3,0) {$H_1(|\calO|;\bbZ)$};

\draw[->>] (OP1) -- (Pi1) node [midway, left] {$\phi$};
\draw[->>] (OP1) -- (Coh) node [midway, above] {$\xi$};
\draw[->>] (Pi1) -- (Hom) node [midway, above] {$h_1$};
\draw[->>] (Coh) -- (Hom) node [midway, left] {$\psi$};
\end{tikzpicture}
\end{center}

where all the maps are surjective, $\phi$ is the canonical map, and $h_1$ is the Hurewicz homomorphism.

Furthermore, if $\phi$ is an isomorphism, so is $\psi$.
\end{mthm}

One interesting aspect of this result is that it provides the first known topological implication of the orbifold fundamental group other than the classical surjection $\pi_1^{orb}(\calO)\to\pi_1(|\calO|)$. The proof of this theorem relies on many classical results for cohomology as well as some understanding of non-manifold points, that is points in $\calO$ whose neighborhoods are not topological disks. As we will show, $H^{n-1},H_1$ form the only obstruction to integer-valued Poincar\'e Duality in dimensions 4 and 5, so we get a corollary:

\begin{mcor}\label{Mthm:PD45}
Let $\calO^n$ be a compact orientable orbifold without boundary, with $n=4$ or $n=5$. Then, if the natural map $\pi_1^{orb}(\calO)\to\pi_1(|\calO|)$ is an isomorphism, $|\calO|$ satisfies integer-valued Poincar\'e Duality. That is, $H^k(|\calO|;\bbZ)=H_{n-k}(|\calO|;\bbZ)$
\end{mcor}

\section{Preliminaries}

Recall that an $n$-dimensional orbifold $\calO^n$ is a space modeled locally on $\bbR^n/\Gamma$ with $\Gamma\subset O(n)$ finite. Given a point $p\in\calO$, the orbifold group at $p$, which we'll denote as $\Gamma_p$ is the subgroup of the group $\Gamma$ in the local chart $\bbR^n/\Gamma$, that fixes a lift of $p$ to $\bbR^n$. Note that different choices of a lift of $p$ result in $\Gamma_p$ being conjugated, and as such, we will think of $\Gamma_p$ up to conjugacy.

In parts of this paper, it will be useful to think of an orbifold $\calO^n$ as a disjoint collection of connected strata. Each stratum is a connected component of points with the same (up to conjugacy) orbifold group. One stratum deserves special mention, $\calO^{reg}$ is the stratum of points with trivial orbifold groups, and points in $\calO^{reg}$ are called regular. Furthermore $\calO^{reg}$ is an open dense subset of $\calO$. Recall that an orbifold $\calO$ is said to be orientable if $\calO^{reg}$ is orientable, and each orbifold group $\Gamma_p$ preserves orientation, that is $\Gamma_p\subset SO(n)$ for each $p\in\calO$. A choice of an orientation on $\calO$ is a choice of an orientation on $\calO^{reg}$.

We call an orbifold $\calU$ a cover of $\calO$ if $\calO = \calU/\Gamma$, where $\Gamma$ acts discretely, and the action of $\Gamma$ preserves the orbifold structure. We recall the definition of $\pi_1^{orb}$, the orbifold fundamental group. If $\calO=\calU/\Gamma$, with $\Gamma$ discrete and $\calU$ admitting no covers, then $\pi_1^{orb}(\calO)=\Gamma$.

\begin{rmk}
There is an alternative definition of an orbifold fundamental group, which is analogous to the path homotopy definition of the topological fundamental group. For precise definition and proof of the equivalence see \cite{SG3M}.
\end{rmk}

A well known result in orbifolds is that there exists a natural surjective map $\phi:\pi_1^{orb}(\calO)\twoheadrightarrow \pi_1(|\calO|)$.

In \cite{SVM} Satake proved $\bbR$-valued Poincar\'e duality for orbifolds:

\begin{prop}
Let $\calO^n$ be a compact orientable orbifold. Then, $H^k(|\calO|;\bbR)=H_{n-k}(|\calO|;\bbR)$.
\end{prop}

The core idea of the proof is the fact that $H^*(S^{n-1}/\Gamma;\bbR)=H^*(S^{n-1};\bbR)$ for any finite $\Gamma\subset SO(n)$. Since this fails for $\bbZ$-valued cohomology, we need a new approach.

\begin{lem}\label{Lem:Cover}
Let $X$ be a path connected space (resp. orbifold) with a path connected subspace (resp. suborbifold) $A$. Let $f:\tilde{X}\to X$ be a path connected cover (resp. orbifold cover) of $X$ such that $\tilde{A}=f^{-1}(A)$ is path connected and $i_*:\pi_1(\tilde{A})\to\pi_1(\tilde{X})$ (resp. $i_*:\pi_1^{orb}(\tilde{A})\to\pi_1^{orb}(\tilde{X})$) is an isomorphism. Then, $i_*:\pi_1(A)\to \pi_1(X)$ (resp. $i_*:\pi_1^{orb}(A)\to \pi_1^{orb}(\partial X)$) is an isomorphism.
\end{lem}

\begin{proof}
The argument in here is presented using regular homotopies, but the same arguments work for the path homotopy definition of the orbifold fundamental group.

We show that the map $i_*:\pi_1(A)\to\pi_1(X)$ is both injective and surjective.

Suppose $\gamma:[0,1]\to A$ is a loop that is null-homotopic in $X$. Then, in particular, the lift of $\gamma$ to $\tilde{\gamma}:[0,1]\to\tilde{A}\subset\tilde{X}$ is also null-homotopic in $\tilde{X}$. Therefore, it must also be null-homotopic in $\tilde{A}$, which implies that $\gamma=f\circ\tilde{\gamma}$ is null-homotopic in $A$. Therefore, $i_*$ is injective.

Now, suppose that $\gamma:[0,1]\to X$ is a loop with $\gamma(0)=\gamma(1)\in A\subset X$. Then, consider some lift of $\gamma$, call it $\tilde{\gamma}:[0,1]\to X$ with $\tilde{\gamma}(0),\tilde{\gamma}(1)\in\tilde{A}$ (possibly distinct).

Let $\tilde{\phi}:[0,1]\to\tilde{A}$ be some path with $\tilde{\phi}(0)=\tilde{\gamma}(0)$ and $\tilde{\phi}(1)=\tilde{\gamma}(1)$. Then we obtain a loop $\tilde{\gamma}\tilde{\phi}^{-1}$ in $\tilde{X}$ based at $\tilde{\gamma}(0)$ by first traversing $\tilde{\gamma}$, and then $\tilde{\phi}$ in reverse.

Let $\tilde{\psi}:[0,1]\to\tilde{A}$ be some loop in $\tilde{A}$ that is homotopic to the loop $\tilde{\gamma}\tilde{\phi}^{-1}$ also based at $\tilde{\gamma}(0)$.

Let $\psi,\phi$ be the projections of $\tilde{\psi},\tilde{\phi}$ respectively onto loops in $A\subset X$. Then, we have $\psi\simeq \gamma\phi^{-1}$, so $\gamma\simeq \psi\phi$, but $\psi\phi$ is a loop in $A$, so we conclude that $i_*:\pi_1(A)\to\pi_1(X)$ is surjective.

Since the argument is based purely on abstract homotopy of loops, the same argument works for $\pi_1^{orb}$.
\end{proof}

We conclude this section with an observation that for an orientable orbifold $\calO$, if $p\in\calO$ is some point such that for $\epsilon>0$ small $|B_\epsilon(p)|$ is not homeomorphic to a disk, then $p$ lies in a stratum of codimension $\geq 4$. This is an immediate consequence of representation theory and the classification of finite subgroups of $SO(3)$. We call these points non-manifold points.

In particular, all orientable orbifolds of dimension $\leq 3$ have underlying spaces homeomorphic to manifolds. In dimension 4, compact orientable orbifolds have at most finitely many non-manifold points. And for compact orientable orbifolds of dimension 5, the set of non-manifold points forms a (possibly disconnected) graph.

\section{Proof of Theorem \ref{Mthm:Pi1Orb}}

In this section we show that there exist surjective maps $\xi:\pi_1^{orb}(\calO)\to H^{n-1}(|\calO|;\bbZ)$ and $\psi:H^{n-1}(|\calO|;\bbZ)\to H_1(|\calO|;\bbZ)$, as well as their compatibility with classical maps.

\begin{lem}\label{Lem:NonMfld}
Let $\calO$ be a compact orientable orbifold. Let $X^c\subset\calO$ be a connected component of the set of non-manifold points in $\calO$. There exists a small neighborhood $N^c$ of $X^c$ such that the map $i_*:\pi_1^{orb}(\partial N^c)\to \pi_1^{orb}(N^c)$ is an isomorphism.
\end{lem}

\begin{proof}
Let $p\in\calO$ be an isolated non-manifold point, then $\pi_1^{orb}(\bar{B_\epsilon(p)}) = \pi_1^{orb}(\partial B_\epsilon(p)) = \Gamma_p$, as follows from considering the local cover of the neighborhood and applying Lemma \ref{Lem:Cover}. This proves the Lemma for the case when $X^c$ is a single point. 

Next suppose that $X^c$ is a graph, that is a connected collection of 1 and 0 dimensional strata. In this case, instead of looking at $N^c$ and $\partial N^c$, we instead focus on specific covers. Namely, each connected graph is covered by a tree, so take such a cover of $X^c$ and the induced covers of $N^c$ and $\partial N^c$. Call these covers $N, \partial N$ respectively. Also, for convenience, we consider $N$ obtained by taking small $\delta_0>\delta_1>0$ and considering $\delta_1$-neighborhoods of the 1-dimensional strata and $\delta_0$-neighborhoods of the 0-dimensional strata, choose $\delta_1$ such that the intersection of the $\delta_1$-neighborhoods lies inside the $\delta_0$-neighborhoods. We now show that $\pi_1^{orb}(\partial N)\to\pi_1^{orb}(N)$ is an isomorphism.

We prove this using van Kampen. Consider the following subsets of $N$: Let $V_i = \bar{N_{\delta_0}(S_i^0)}$ be a closed $\delta_0$-neighborhood of a 0-dimensional stratum of $X^c$ (that is a ``vertex'' of the ``graph''). Let $E_{ij} = \bar{N_{\delta_1}(S_{ij}^1)} \setminus (N_{\delta_0}(S_i^0)\cup N_{\delta_0}(S_j^0))$ be a closed $\delta_1$-neighborhood of a 1-dimensional stratum of $X^c$ (``edge'' $S_{ij}^1$ with endpoints $S_i^0,S_j^0$) with the two endpoints cut out. We now consider pairs: $(V_i, \partial N \cap V_i)$ and $(E_{ij}, \partial N\cap E_{ij})$.

Topologically, $V_i = B^n/\Gamma_i$, and $\partial N \cap V_i = (S^{n-1} \setminus (D_1\cup D_2 \cup \cdots \cup D_n))/\Gamma_i$ where $D_k$ are disks in $S^{n-1}$, in particular, since $n = 5$, we know that $\pi_1(S^{n-1} \setminus (D_1\cup D_2 \cup \cdots \cup D_n)) = 0$, so $\pi_1^{orb}(N\cap V_i) = \Gamma_i = \pi_1^{orb}(V_i)$. Similarly, $E_{ij} = I \times B^{n-1}/\Gamma_{ij}$ and $\partial N \cap E_{ij} = I\times S^{n-2}/\Gamma_{ij}$ both have the same orbifold fundamental group. Moreover, $V_i\cap E_{ij} = B^{n-1}/\Gamma_{ij}$ and $N\cap V_i\cap E_{ij} = S^{n-2}/\Gamma_{ij}$ also have identical orbifold fundamental groups.

Now consider two van Kampen constructions:
\[
N = \left(\bigcup\limits_i V_i\right) \cup \left(\bigcup\limits_{ij} E_{ij}\right)
\]
and
\[
\partial N = \left(\bigcup\limits_i (V_i\cap \partial N) \right) \cup \left(\bigcup\limits_{ij} (E_{ij}\cap \partial N)\right).
\]

Since each piece and each intersection in the two constructions has the same $\pi_1^{orb}$, and since each of these isomorphisms is natural, van Kampen implies that $i_*:\pi_1^{orb}(\partial N)\to \pi_1^{orb}(N)$ is an isomorphism. So, by Lemma \ref{Lem:Cover}, $\pi_1^{orb}(\partial N^c)\to\pi_1^{orb}(N^c)$ is also an isomorphism.

If $X^c$ contains strata of dimension $\geq 2$, we may need to introduce some \textit{pseudo-strata}. By a pseudo-stratum we understand a component of a stratum that has lower dimension than the rest of the stratum.

The goal of pseudo-strata is to create a nice cellular decomposition of $X^c$ in a way that for each cell of dimension $\geq 2$, its boundary is connected.

The work above demonstrates that the 1-skeleton of $X^c$ satisfies the lemma. We then proceed to glue in the 2-cells by taking a $\delta_2>0$ neighborhood of each 2-cell such that the intersection of any two such neighborhoods lies in the already constructed neighborhood of the 1-skeleton. By van Kampen, when we conduct this gluing, $i_*:\pi_1^{orb}(\partial N^c)\to \pi_1^{orb}(N^c)$ remains an isomorphism.

Repeat the construction for all the cells of dimension $\geq 2$.

\end{proof}

In the above proof of the 1-dimensional case we passed to a graph to avoid issues with disconnected intersection of components when applying van Kampen. The pseudo-strata were used for the same reason.

\begin{cor}\label{Cor:NonMfld}
Let $\calO$ be a compact orientable orbifold. Let $X\subset\calO$ be the set of non-manifold points in $\calO$. Then, there exists $X\subset N\subset\calO$ with $N$ open such that the map $i_*:\pi_1^{orb}(\calO\setminus N) \to \pi_1^{orb}(\calO)$ is an isomorphism.
\end{cor}

This is an immediate consequence of Lemma \ref{Lem:NonMfld} and van Kampen. If $X$ consists of multiple connected components, we use Lemma \ref{Lem:NonMfld} repeatedly to cut them out one at a time without changing the orbifold fundamental group of the remaining portion.

Take the definitions of $X$ and $N$ as above. Let $\calU = \calO\setminus N$. We define a map from $\pi_1^{orb}(\calO)\to H^{n-1}(|\calO|;\bbZ)$ as the following composition:

\begin{align*}
\pi_1^{orb}(\calO) &\xrightarrow{\cong} \pi_1^{orb}(\calU) \twoheadrightarrow \pi_1(|\calU|) \twoheadrightarrow H_1(|\calU|;\bbZ) \xrightarrow{\cong} H^{n-1}(|\calU|,|\partial\calU|;\bbZ)\\
&\xrightarrow{\cong} H^{n-1}(|\calO|,|N|;\bbZ) \xrightarrow{\cong} H^{n-1}(|\calO|,X;\bbZ) \xrightarrow{\cong} H^{n-1}(|\calO|;\bbZ).
\end{align*}

The maps come from the following results in order:

\begin{enumerate}
\item Corollary \ref{Cor:NonMfld}

\item Classical orbifold result

\item Hurewicz

\item Poincar\'e-Lefschetz Duality

\item Excision

\item Retraction of $|N|$ onto $X$

\item Long exact sequence for relative cohomology
\end{enumerate}

We now utilize this construction to construct and control the map $H^{n-1}(|\calO|)\to H_1(|\calO|)$. In particular, we use the surjective maps $\pi_1^{orb}(\calO)\to H_1(|\calO|)$ and $\pi_1^{orb}(\calO)\to H^{n-1}(|\calO|)$.

Consider $\calU$ as defined in the proof of Corollary \ref{Cor:NonMfld}, that is $\calO$ with a neighborhood of the non-manifold points cut out.

In general, the map $\pi_1(|\calU|)\to\pi_1(|\calO|)$ is surjective, since $N = \calO\setminus\calU$ retracts onto a subset of codimension $\geq 4$. This implies that $H_1(|\calU|)\to H_1(|\calO|)$ is also surjective, so, since $H_1(|\calU|) \cong H^{n-1}(|\calO|)$, we get a surjective map $H^{n-1}(|\calO|;\bbZ)\to H_1(|\calO|;\bbZ)$.

However, if we assume that $\pi_1^{orb}(\calO)\to \pi_1(|\calO|)$ is an isomorphism, then we get $\pi_1^{orb}(\calO) \xrightarrow{\cong} \pi_1^{orb}(\calU) \twoheadrightarrow \pi_1(|\calU|) \twoheadrightarrow \pi_1(|\calO|)$ must be a composition of isomorphisms. Therefore, $H^{n-1}(|\calO|;\bbZ)$ must be isomorphic to $H_1(|\calO|;\bbZ)$.

This completes the proof of Theorem \ref{Mthm:Pi1Orb}.

\section{Proof of Corollary \ref{Mthm:PD45}}

In this section all of our orbifolds are assumed to be compact and orientable. We begin by considering the general structure of cohomology of 4 and 5 dimensional orbifolds.

\subsection{Structure of $H^*(|\calO^4|;\bbZ)$}

A four-dimensional orbifold has the following $\bbR$-valued cohomology groups:
\[
H^k(|\calO|;\bbR) = \begin{cases}
\bbR & k=0,4\\
\bbR^m & k=1,3\\
\bbR^n & k=2.
\end{cases}
\]
Together with classical homology and cohomology results, this implies that
\[
H^k(|\calO|;\bbZ) = \begin{cases}
\bbZ & k=0\\
\bbZ^m & k=1\\
\bbZ^n + \tau^2 & k=2\\
\bbZ^m + \tau^3 & k=3\\
\bbZ + \tau^4 & k=4,
\end{cases}
\]
where $\tau^k$ is the torsion component of $H^k$. The homology groups are:
\[
H_k(|\calO|;\bbZ) = \begin{cases}
\bbZ & k=0,4\\
\bbZ^m + \tau^2 & k=1\\
\bbZ^n + \tau^3 & k=2\\
\bbZ^m + \tau^4 & k=3.
\end{cases}
\]

Let $p\in\calO$ be a regular point, $B_\epsilon(p)\subset\calO^{reg}$ be a small neighborhood of $p$, and $X=|\calO|\setminus B_\epsilon(p)$. Then, $|\calO| = \bar{B_\epsilon(p)}\cup_{S^3} X$. Then, $H^4(X;\bbZ)=0$ since $X$ deformation retracts onto its 3-skeleton. Using Mayer-Vietoris, we get $H^3(S^3)\to H^4(|\calO|)\to H^4(X)$ ($\bbZ\to\bbZ+\tau^4\to 0$) is exact. Therefore, $\tau^4=0$. We conclude that only obstruction to Poincar\'e Duality is the difference between $\tau^2$ and $\tau^3$.

\subsection{Structure of $H^*(|\calO^5|;\bbZ)$}

We repeat the previous arguments to obtain information about the cohomology of five-dimensional orbifolds.

The general structure is
\[
H^k(|\calO|;\bbR) = \begin{cases}
\bbR & k=0,5\\
\bbR^m & k=1,4\\
\bbR^n & k=2,3.
\end{cases}
\]
Together with classical results, we conclude that
\[
H^k(|\calO|;\bbZ) = \begin{cases}
\bbZ & k=0\\
\bbZ^m & k=1\\
\bbZ^n + \tau^2 & k=2\\
\bbZ^n + \tau^3 & k=3\\
\bbZ^m + \tau^4 & k=4\\
\bbZ + \tau^5 & k=5,
\end{cases}
\]
which gives us the following homology groups:
\[
H_k(|\calO|;\bbZ) = \begin{cases}
\bbZ & k=0,5\\
\bbZ^m + \tau^2 & k=1\\
\bbZ^n + \tau^3 & k=2\\
\bbZ^n + \tau^4 & k=3\\
\bbZ^m + \tau^5 & k=4.
\end{cases}
\]

The same argument as we employed in the four-dimensional case is used to show that $\tau^5=0$. So, the only obstruction to Poincar\'e Duality is the difference between $\tau^2$ and $\tau^4$.

\subsection{Consequences}

Note that in dimensions 4 and 5, the only obstruction was in the difference of torsion between $H_1(|\calO|)$ and $H^{n-1}(|\calO|)$. Now with the assumption that $\pi_1^{orb}(\calO)\to \pi_1(|\calO|)$ is an isomorphism, we utilize Theorem \ref{Mthm:Pi1Orb} to conclude that $\bbZ$-valued Poincar\'e Duality holds, which complete the proof of Corollary \ref{Mthm:PD45}.

\begin{rmk}\label{Rmk:HighDim}
Starting in dimension 6, there are multiple pairs of torsion groups that may obstruct Poincar\'e Duality. In particular, in dimension 6, we have $\tau^2$ vs $\tau^5$ and $\tau^3$ vs $\tau^4$.

In fact, in dimension $n$, there are precisely $\left\lfloor\dfrac{n-2}{2}\right\rfloor$ such pairs of torsion groups. Theorem \ref{Mthm:Pi1Orb} allows us to control the difference in one pair.
\end{rmk}

\begin{cor}
Weighted projective spaces $\CP^2[\lambda_0,\lambda_1,\lambda_2]$ and the family of orbifolds $SO(3)\backslash SU(3)/S^1_{p,q}$ introduced by the author in \cite{YSU3} have underlying spaces with integer (co)homology of $\CP^2$.
\end{cor}

Both families of orbifolds are simply connected, as they are base spaces of orbi-fiber bundles with a simply connected total space and a connected fiber. Furthermore, an adaptation of an argument due to Kobayashi (see \cite{KFix} and \cite{KTrans}) to the orbifold case shows that both families have Euler characteristic 3. For weighted projective spaces this result can also be obtained from iterated application of Mayer-Vietoris.

\bibliographystyle{amsalpha}
\bibliography{References}
\end{document}